\documentclass[a4paper,11pt]{article}
\usepackage[margin=1in]{geometry}
\usepackage{authblk,amsthm}
\usepackage{tikz}
\usepackage{subcaption}
\usepackage{amssymb}
\usepackage{stmaryrd}
\usepackage{amsmath,mathtools}
\usepackage{commath}
\usepackage{hyperref}
\hypersetup{
  colorlinks   = true, %
  urlcolor     = blue, %
  linkcolor    = blue, %
  citecolor   = blue %
}
\usepackage{cleveref}

\newtheorem{lemma}{Lemma}
\newtheorem{corollary}{Corollary}
\newtheorem{remark}{Remark}

\newcommand{\jump}[1]{\llbracket #1 \rrbracket}
\newcommand{\av}[1]{\{\!\!\{#1\}\!\!\}}
\usepackage{xspace,color}

\title{An algebraic
  preconditioner for the exactly divergence-free discontinuous
  Galerkin method for Stokes}
\author[$\dagger$]{Sander Rhebergen}
\affil[$\dagger$]{Department of Applied Mathematics, University of Waterloo,
  Waterloo, ON N2L~3G1, Canada; \href{mailto:srheberg@uwaterloo.ca}{srheberg@uwaterloo.ca}.}

\author[$\ddag$]{Ben S. Southworth}
\affil[$\ddag$]{Los Alamos National Laboratory, Los Alamos NM, U.S.A.;
  \href{mailto:southworth@lanl.gov}{southworth@lanl.gov}}
\begin{document}
\maketitle
\begin{abstract}
  We present an optimal preconditioner for the exactly divergence-free
  discontinuous Galerkin (DG) discretization of Cockburn, Kanschat,
  and Sch\"otzau [\textit{J. Sci. Comput.}, 31 (2007), pp. 61--73] and
  Wang and Ye [\textit{SIAM J. Numer. Anal.}, 45 (2007),
  pp. 1269--1286] for the Stokes problem. This DG method uses finite
  elements that use an $H({\rm div})$-conforming basis, thereby
  significantly complicating its solution by iterative
  methods. Several preconditioners for this Stokes discretization have
  been developed, but each is based on specialized solvers or
  decompositions, and do not offer a clear framework to generalize to
  Navier--Stokes. To avoid requiring custom solvers, we hybridize the
  $H({\rm div})$-conforming finite element so that the velocity lives
  in a standard $L^2$-DG space, and present a simple algebraic
  preconditioner for the extended hybridized system. The proposed
  preconditioner is optimal in $h$, super robust in element order
  (demonstrated up to 5th order), effective in 2d and 3d, and only
  relies on standard relaxation and algebraic multigrid methods
  available in many packages. The hybridization also naturally extends
  to Navier--Stokes, providing a potential pathway to effective
  black-box preconditioners for exactly divergence-free DG
  discretizations of Navier--Stokes.
\end{abstract}
\section{Introduction}
\label{sec:introduction}

In this paper we present a preconditioner for the exactly divergence
free discontinuous Galerkin (DG) discretization
\cite{Cockburn:2007b,Wang:2007} of the Stokes equations. Given an
external body force $\boldsymbol{f} : \Omega \to \mathbb{R}^d$ defined
on a domain $\Omega$ in $d=2$ or $d=3$ dimensions, the Stokes
equations for the velocity $\boldsymbol{u} : \Omega \to \mathbb{R}^d$
and (kinematic) pressure $p : \Omega \to \mathbb{R}$ are given by:
\begin{equation}
  \label{eq:stokes}
  -\nabla^2\boldsymbol{u} + \nabla p = \boldsymbol{f} \text{ in } \Omega \qquad \text{ and } \qquad
  \nabla\cdot \boldsymbol{u} = 0 \text{ in } \Omega.
\end{equation}
Given $\boldsymbol{g} : \partial \Omega \to \mathbb{R}^d$, we impose
$\boldsymbol{u} = \boldsymbol{g}$ on $\partial\Omega$, the boundary of
the domain, and require that the pressure has zero mean.

The exactly divergence-free DG method of Cockburn, Kanschat, and
Sch\"otzau \cite{Cockburn:2007b} and Wang and Ye \cite{Wang:2007}
results in a discretization for incompressible flows that is locally
conservative, energy-stable, optimally convergent, and pressure-robust
\cite{John:2017}. To achieve this they use finite elements that use an
$H({\rm div})$-conforming basis such as the Raviart--Thomas (RT)
\cite{Raviart:1977} and Brezzi--Douglas--Marini (BDM)
\cite{Brezzi:1985} elements.

Since its introduction, this discretization has been applied in
different forms to many problems, including the Navier--Stokes
equations \cite{Cockburn:2007b}, Stokes equations \cite{Wang:2007},
the Boussinesq equations \cite{Oyarzua:2014}, incompressible Euler
equations \cite{Guzman:2017}, coupled Stokes--Darcy flow
\cite{Kanschat:2010}, and incompressible magnetohydrodynamics problems
\cite{Greif:2010}.  Unfortunately, custom solvers are necessary when
dealing with $H({\rm div})$-conforming finite elements, and relatively
few have been developed for incompressible flow.
Even a standard ``model'' problem like the (vector) Laplacian requires
significant care when posed in $H({\rm div})$, and many standard
parallel preconditioners such as multigrid or algebraic multigrid (AMG)
do not work out-of-the-box. 

Auxiliary space preconditioning techniques have been developed for 
$H({\rm div})$ and $H({\rm curl})$ \cite{Hiptmair.2007}, but most of
these methods are developed specifically for bi-linear forms
$({\rm div} u, {\rm div} v)$. Auxiliary space preconditioning
specifically designed for DG discretizations of Stokes was
introduced in \cite{Dios.2014}. In 2d, the preconditioning requires
solving four scalar Laplacian problems. According to \cite{Dios.2014},
extensions to 3d require an additional solution of a mixed
discretization of the Laplacian, as well as an approximation
via lumping mass matrices in defining orthogonal projections within
the preconditioner; however, no 3d implementation or numerical
results are provided in \cite{Dios.2014}. Aside from the more 
complex extension to 3d, a more notable downside of the 
auxiliary-space approach is that it does not provide a clear path
to generalize to Navier--Stokes, as the underlying theory
and framework rely heavily on symmetry and definiteness.

To our knowledge, the only other effective preconditioners
for divergence-free DG discretizations of the Stokes equation are
geometric multigrid methods with specialized overlapping block
smoothers for $H({\rm div})$ \cite{Kanschat:2015,Kanschat.2016,Hong.2016}. 
Such methods have proven effective and robust on Stokes, but also have
several drawbacks. For one, geometric multigrid methods with custom
smoothers are far from ``black-box,'' typically requiring nontrivial software
infrastructure to implement or to apply with existing libraries. In
addition, extensions to Navier--Stokes with moderate or strong
advection is challenging. Geometric multigrid for problems with
advection generally requires line relaxation and/or semi-coarsening
(generally both or alternating line relaxation for non-grid-aligned
advection) \cite{roberts2004textbook,brown2000semicoarsening}. Each
of these immediately limit one to structured grids (particularly in any 
parallel setting); moreover, line relaxation scales poorly in parallel,
and its 3d analogue ``plane'' relaxation is challenging or sometimes
infeasible to implement on more complex geometries \cite{stuben2001review}. 

Thus, although there have been several robust preconditioners developed
for divergence-free DG discretizations of Stokes, there remain several
gaps in the literature. First, to our knowledge there have been no
effective parallel preconditioners (that is, either an all-at-once approach
or block preconditioning with effective inner solvers as well) proposed
for divergence-free DG
discretizations of Navier--Stokes. Moreover, none of the Stokes
methods currently available offer a natural pathway to generalize to
large-scale simulations of Navier--Stokes.\footnote{Of course
if one considers problems small enough to fit on one processor, alternating
line relax or ordered Gauss--Seidel that follows the characteristics are
both feasible and likely effective.} Second, even for Stokes,
existing methods are relatively intrusive in terms of implementation.
This is okay in principle, but it is often beneficial to have
robust methods that are straightforward to implement using existing
technology and widely used software libraries.

The main objective of this paper is to present a preconditioner for
divergence-free DG discretizations of Stokes that (i) is based on
standard ``black-box'' components, and (ii) provides a potential
framework to generalize to Navier--Stokes. To achieve this we
hybridize the BDM element following \cite{Arnold:1985} by introducing
a Lagrange multiplier to enforce the normal continuity of the velocity
approximation across element boundaries. As a result, the velocity of
the hybridized system lives in $L^2$, and standard preconditioners
for DG discretizations of (advection-)diffusion can then be applied
to solve the velocity block. This is a similar principle as proposed
in \cite{Lee:2017} for solving grad-div problems in $H({\rm div})$.
We couple this with the introduction of
a simple preconditioner for the hybridized system, which we demonstrate
to be robust in $h$ and element order, and naturally applicable in
2d and 3d.

Solvers for hybridized systems were studied also in
\cite{Dobrev:2019}, however they studied problems in
$H({\rm div})$-spaces discretized by conforming elements. In the case
of the divergence-free DG method, the BDM element is non-conforming
for the Stokes equations. A consequence is that, unlike
\cite{Dobrev:2019}, we are unable to reduce the hybridized system to a
global problem for the Lagrange multipliers only (this would require a
hybridizable DG (HDG) method which we do not discuss here; see, e.g.,
\cite{Cockburn:2014b,Lehrenfeld:2016,Rhebergen:2018a}). Instead, we
obtain a saddle-point problem for the velocity, pressure, and Lagrange
multiplier unknowns. We will show that the pressure/Lagrange
multiplier Schur complement is spectrally equivalent to a $2 \times 2$
block diagonal matrix consisting of differently weighted pressure and
Lagrange multiplier mass-matrices. We will demonstrate that weighing
the pressure mass-matrix more heavily than the Lagrange multiplier
mass-matrix significantly outperforms a preconditioner in which both
mass-matrices are weighted equally (we remark that the preconditioner
for the pressure/Lagrange multiplier Schur complement based on equally
weighting the pressure and Lagrange multiplier mass-matrices is the
same as the Schur complement approximation analyzed for a
divergence-free HDG method for Stokes in
\cite{Rhebergen:2018b}). Combining this with standard AMG for the
$L^2$-DG discretization of the vector Laplacian results in an optimal
preconditioner for the divergence-free DG discretization of Stokes.

The remainder of this paper is organized as follows. In
\cref{s:hdiv-dg-fem} we summarize the exactly divergence-free DG
method for Stokes and its hybridization. The preconditioner is
presented and analyzed in \cref{s:precon} and its performance is
discussed in \cref{s:numex}. Conclusions are drawn in
\cref{sec:conclusions}.

\section{The exactly divergence-free DG method for Stokes}
\label{s:hdiv-dg-fem}

\subsection{Notation}

Let $\mathcal{T}$ be a triangulation of mesh size~$h_{K}$ of the
domain $\Omega$ into simplices~$\{K\}$. We denote the set of all
interior and boundary facets of $\mathcal{T}$ by $\mathcal{F}_I$ and
$\mathcal{F}_B$, respectively, and the set of all facets is denoted by
$\mathcal{F} = \mathcal{F}_I \cup \mathcal{F}_B$. For element $K$ and
boundary $\partial K$, the outward unit normal vector is denoted
by~$\boldsymbol{n}$. On an interior facet shared by elements $K^+$ and
$K^-$, the average $\av{\cdot}$ and jump $\jump{\cdot}$ operators are
defined in the usual way, i.e, for a scalar $q$
\begin{equation}
  \av{q} := \tfrac{1}{2}\del{q^+ + q^-}, \qquad
  \jump{q\boldsymbol{n}} := q^+\boldsymbol{n}^+ + q^-\boldsymbol{n}^-,
\end{equation}
where we remark that $\boldsymbol{n}^+ = -\boldsymbol{n}^-$. On
boundary facets we set $\av{q} := q$ and
$\jump{q\boldsymbol{n}} := q\boldsymbol{n}$. Average and jump
operators for vectors and tensors are defined similarly.

The finite element space for the pressure is defined as
\begin{equation}
  \label{eq:pressure_fs}
  Q_h^k := \cbr[1]{ q_h \in L^2(\mathcal{T}),\ q_h \in P_{k}(K)\
    \forall K\in\mathcal{T}} \cap L^2_0(\Omega),
\end{equation}
where $P_k(K)$ denotes the space of polynomials of degree $k > 0$ on
element $K$. For the velocity approximation, we consider the
Brezzi--Douglas--Marini (BDM) function spaces. Denote BDM spaces of
order $k$ by ${\rm {\bf BDM}}_h^k$. The pair
${\rm {\bf BDM}}_h^k\backslash Q_h^{k-1}$ forms an inf-sup stable
finite element pair that furthermore has the desirable property that
$\nabla\cdot {\rm {\bf BDM}}_h^k = Q_h^{k-1}$.

We denote the $L^2$-norm on an element $K \in \mathcal{T}$, a facet
$F \in \mathcal{F}$, and domain $\Omega$ by, respectively,
$\norm{\cdot}_K$, $\norm{\cdot}_F$, and $\norm{\cdot}$.  In what
follows we will also require the following norm for functions
$\boldsymbol{v} \in {\rm {\bf BDM}}_h^k$:
\begin{equation}
  \norm{\boldsymbol{v}}_{1,h}^2 := \sum_{K \in \mathcal{T}} \norm{\nabla \boldsymbol{v}}_{K}^2
  + \sum_{F \in \mathcal{F}} \int_F h_F^{-1}|\jump{\boldsymbol{v} \otimes \boldsymbol{n}}|^2 \dif s.
\end{equation}

\subsection{Discretization}

The exactly divergence-free symmetric interior penalty (SIP)-DG method
of \cite{Cockburn:2007b,Wang:2007} applied to the Stokes equations
\cref{eq:stokes} is given by: find
$(\boldsymbol{u}_h, p_h) \in {\rm {\bf BDM}}_h^k \times Q_h^{k-1}$
such that
\begin{equation}
  \label{eq:DGmethod}
  a_h(\boldsymbol{u}_h, \boldsymbol{v}_h)
  + \tilde{b}_h(p_h, \boldsymbol{v}_h) + \tilde{b}_h(q_h, \boldsymbol{u}_h)
  = \tilde{l}_h(\boldsymbol{v}_h), \quad \forall (\boldsymbol{v}_h, q_h)
  \in {\rm {\bf BDM}}_h^k \times Q_h^{k-1},
\end{equation}
where the bi-linear forms are defined as
\begin{subequations}
  \begin{align}
    \label{eq:bilinforms_a}
    a_h(\boldsymbol{u}, \boldsymbol{v}) 
    =&
       \sum_{K\in\mathcal{T}}\int_K \nabla \boldsymbol{u} : \nabla \boldsymbol{v} \dif x
       + \sum_{F\in\mathcal{F}} \int_F \frac{\eta}{h_F}\jump{\boldsymbol{u}}\cdot\jump{\boldsymbol{v}}\dif s
    \\ \nonumber
     &- \sum_{F\in\mathcal{F}}\int_F \jump{\boldsymbol{u}}\otimes \boldsymbol{n}^+ : \av{\nabla \boldsymbol{v}} \dif s
       - \sum_{F\in\mathcal{F}} \int_F \jump{\boldsymbol{v}}\otimes \boldsymbol{n}^+ : \av{\nabla \boldsymbol{u}} \dif s,
    \\
    \label{eq:bilinforms_b}
    \tilde{b}_h(p, \boldsymbol{v})
    =& -\sum_{K\in\mathcal{T}} \int_K p \nabla \cdot \boldsymbol{v} \dif x,
  \end{align}
  \label{eq:bilinforms}
\end{subequations}
with $\eta>0$ a penalty parameter, and the linear form is defined as
\begin{equation}
  \tilde{l}_h(\boldsymbol{v}) =
  \sum_{K\in\mathcal{T}} \int_K \boldsymbol{f} \cdot \boldsymbol{v} \dif x
  - \sum_{F\in\mathcal{F}_B}\int_F \pd{\boldsymbol{v}}{\boldsymbol{n}} \cdot \boldsymbol{g} \dif s
  + \sum_{F\in\mathcal{F}_B} \int_F  \frac{\eta}{h_F} \boldsymbol{g} \cdot \boldsymbol{v} \dif s.
\end{equation}

We now summarize some properties of the above bi-linear forms (see
\cite{Cockburn:2004b,Cockburn:2007b,Wang:2007}). First, there exist
constants $\alpha > 0$ and $c_a > 0$, independent of $h$, such that if
$\eta$ is sufficiently large then for all
$\boldsymbol{u}_h, \boldsymbol{v}_h \in {\rm {\bf BDM}}_h^k$,
\begin{equation}
  \label{eq:stab_bound_ah}
  a_h(\boldsymbol{v}_h, \boldsymbol{v}_h) \ge \alpha\norm{\boldsymbol{v}_h}_{1,h}^2
  \quad\text{and}\quad
  \envert[1]{a_h(\boldsymbol{u}_h, \boldsymbol{v}_h)} \le c_a\norm{\boldsymbol{u}_h}_{1,h}\norm{\boldsymbol{v}_h}_{1,h}.
\end{equation}
Furthermore, there exist constants $c_b > 0$ and $\beta > 0$,
independent of $h$, such that for all
$\boldsymbol{v}_h \in {\rm {\bf BDM}}_h^k$ and for all
$q_h \in Q_h^{k-1}$,
\begin{equation}
  \label{eq:bound_bh}
  \envert[1]{\tilde{b}_h(q_h, \boldsymbol{v}_h)} \le c_b \norm{\boldsymbol{v}_h}_{1,h}\norm{q_h}
  \quad \text{and} \quad
  \beta\norm{q_h} \le \sup_{\boldsymbol{v}_h \in {\rm {\bf BDM}}_h^k}
    \frac{ \tilde{b}_h(q_h, \boldsymbol{v}_h) }{\norm{\boldsymbol{v}_h}_{1,h}}.
\end{equation}

\subsection{Hybridized discretization}

Standard multigrid solvers are ineffective on the BDM element
\cite{Arnold:2000}. This has led to the creation of custom geometric
multigrid methods with overlapping block smoothers for BDM elements
\cite{Kanschat:2015,Kanschat.2016,Arnold:2000,Hong.2016}. Here we
pursue a different approach, hybridizing the discretization as in
\cite{Dobrev:2006}, so that the velocity is approximated by functions
in an $L^2$ space rather than functions in an $H({\rm div})$ space,
and standard multigrid and AMG methods can be applied. For this, let
us introduce, respectively, the following DG velocity space and
Lagrange multiplier space:
\begin{subequations}
  \begin{align}
    \boldsymbol{V}_h^k
    &:= \cbr[0]{\boldsymbol{v}_h \in \boldsymbol{L}^2(\Omega):\ \boldsymbol{v}_h \in P_k(K)^d\ \forall K \in \mathcal{T}},
    \\
    M_h^k
    &:= \cbr[0]{\xi_h \in L^2(\mathcal{F}):\ \xi_h \in P_k(F)\ \forall F \in \mathcal{F}},
  \end{align}
\end{subequations}
For functions in the Lagrange multiplier space, $\xi \in M_h$, we
define
\[
  \norm{\xi}_{\lambda} := \del[0]{\sum_{F \in \mathcal{F}} h_F
  \norm[0]{\xi}_F^2}^{1/2}.
\]

We then obtain the extended problem: find
$(\boldsymbol{u}_h, p_h, \lambda_h) \in \boldsymbol{V}_h^k \times
Q_h^{k-1} \times M_h^k$ such that
\begin{equation}
  \label{eq:hybridized}
  a_h(\boldsymbol{u}_h, \boldsymbol{v}_h) + b_h((p_h,\lambda_h), \boldsymbol{v}_h)
  + b_h((q_h,\xi_h), \boldsymbol{u}_h)
  = l_h(\boldsymbol{v}_h, \xi_h),
\end{equation}
for all
$(\boldsymbol{v}_h, q_h, \xi_h) \in \boldsymbol{V}_h^k \times
Q_h^{k-1} \times M_h^k$, where
\begin{subequations}
  \begin{align}
    b_h((p,\lambda), \boldsymbol{v})
    &= \tilde{b}_h(p, \boldsymbol{v}) + \sum_{F \in \mathcal{F}}\int_F \jump{\boldsymbol{v} \cdot \boldsymbol{n}}\lambda \dif s,
    \\
    l_h(\boldsymbol{v}, \xi)
    &= \tilde{l}_h(\boldsymbol{v}) + \sum_{F \in \mathcal{F}_B} \int_F \boldsymbol{g} \cdot \boldsymbol{n} \xi \dif s.
  \end{align}  
\end{subequations}
The coercivity and boundedness of $a_h$ \eqref{eq:stab_bound_ah},
also holds for functions $\boldsymbol{u}_h, \boldsymbol{v}_h \in
\boldsymbol{V}_h^k$ (see, e.g., \cite[Chapter 6]{Pietro:book}).
Furthermore, by the Cauchy--Schwarz inequality,
\begin{equation}
  \label{eq:bhextbounded}
  \envert[1]{b_h((q_h, \xi_h), \boldsymbol{v}_h)} \le c_b \norm{\boldsymbol{v}_h}_{1,h}
  \del[0]{\norm{q_h}^2 + \norm{\xi_h}_{\lambda}^2}^{1/2},
\end{equation}
while a minor modification of \cite[Lemma 3]{Rhebergen:2018b} shows
that there exists a constant $\beta_{\lambda} > 0$ such that for all
$\xi_h \in M_h^k$:
\begin{equation}
  \label{eq:infsuplagrange}
  \beta_{\lambda} \norm{\xi_h}_{\lambda}
  \le \sup_{\boldsymbol{v}_h \in \boldsymbol{V}_h^k}
  \frac{ \sum_{F \in \mathcal{F}}\int_F \jump{\boldsymbol{v}_h \cdot \boldsymbol{n}}\xi_h \dif s}{\norm{\boldsymbol{v}_h}_{1,h}}.
\end{equation}
By \cite[Theorem 3.1]{Howell:2011} the inf-sup conditions in
\cref{eq:bound_bh} and \cref{eq:infsuplagrange} may be combined to
yield
\begin{equation}
  \label{eq:infsupbh}
  c_i \del[0]{\beta^2\norm{q_h}^2 + \beta_{\lambda}^2\norm{\xi_h}_{\lambda}^2}^{1/2} \le \sup_{\boldsymbol{v}_h \in \boldsymbol{V}_h^k}
  \frac{ b_h((q_h, \xi_h), \boldsymbol{v}_h) }{\norm{\boldsymbol{v}_h}_{1,h}}
  \quad \forall (q_h, \xi_h) \in Q_h^{k-1} \times M_h^k,
\end{equation}
with $c_i$ a positive constant independent of $h$.

We end this section by noting that the velocity and pressure solution
to the hybridized system \cref{eq:hybridized} is also the velocity and
pressure solution to \cref{eq:DGmethod}. As such, the velocity
solution to the hybridized system \cref{eq:hybridized} is exactly
divergence-free.

\section{Preconditioning}
\label{s:precon}

We now analyze a preconditioner for the extended system
\cref{eq:hybridized}.
Let $u \in \mathbb{R}^{n_u}$, $p \in \mathbb{R}^{n_p}$,
$\lambda \in \mathbb{R}^{n_l}$ be the vectors of the coefficients
associated with $\boldsymbol{u}_h \in \boldsymbol{V}_h^k$,
$p_h \in Q_h^{k-1}$, and $\lambda_h \in M_h^k$. We then define the
matrices $A$, $B$, and $C$ by
$a_h(\boldsymbol{v}_h, \boldsymbol{v}_h) = v^TAv$,
$b_h((q_h,0), \boldsymbol{v}_h) = q^T B v$, and
$b_h((0, \xi_h), \boldsymbol{v}_h) = \xi^T C v$,
for $v \in \mathbb{R}^{n_u}$, $q \in \mathbb{R}^{n_p}$,
$\xi \in \mathbb{R}^{n_l}$. The extended system \cref{eq:hybridized}
can now be written in matrix form as
\begin{equation}
  \label{eq:matsystem}
  \mathbb{A}{\bf x} = {\bf f} \leftrightarrow
  \begin{bmatrix}
    A & B^T & C^T \\
    B & 0 & 0 \\
    C & 0 & 0
  \end{bmatrix}
  \begin{bmatrix}
    u \\ p \\ \lambda
  \end{bmatrix}
  =
  \begin{bmatrix}
    f \\ 0 \\ g
  \end{bmatrix},
\end{equation}
where $f$ and $g$ are the vectors associated with, respectively,
$l_h(\boldsymbol{v}_h,0)$ and $l_h(0,\xi_h)$. Let $\mathcal{B}
\coloneqq [B^T\ C^T]^T$. We will furthermore denote
the negative Schur complement of the block $A$ in $\mathbb{A}$ by
\begin{equation}\label{eq:schur}
  \mathcal{S} = \mathcal{B}A^{-1}\mathcal{B}^T.
\end{equation}

To define the preconditioner we introduce the mass matrices $Q$
and $M$, which are defined by $\norm{q_h}^2 = q^TQq$ and
$\norm{\xi_h}_{\lambda}^2 = \xi^T M \xi$, and set
\begin{equation}\label{eq:mass-prec}
  \mathcal{M} \coloneqq {\rm bdiag}(\omega_q Q, \omega_m M),
\end{equation}
with $\omega_q$
and $\omega_m$ positive constants. The following Lemma proves a 
spectral equivalence between $\mathcal{M}$ and $\mathcal{S}$, by
choosing constants $\omega_q$ and $\omega_m$ based on inf-sup
parameters $\beta$ \eqref{eq:bound_bh} and $\beta_\lambda$
\eqref{eq:infsuplagrange}.

\begin{lemma}
  \label{lem:specEquivnew}  
  Let $\alpha$, $c_a$, $c_b$, $c_i$, $\beta$, and $\beta_\lambda$ be
  the constants given by
  \cref{eq:stab_bound_ah,eq:bhextbounded,eq:infsupbh} and let
  $c_I = c_i\min(\beta,\beta_{\lambda})$. Let
  $\omega_q \coloneqq \beta^2/\min(\beta^2,\beta_{\lambda}^2)$ and
  $\omega_m \coloneqq
  \beta_{\lambda}^2/\min(\beta^2,\beta_{\lambda}^2)$, and define
  $\mathcal{M}$ as in \eqref{eq:mass-prec}. Then $\mathcal{M}$ is
  spectrally equivalent to $\mathcal{S}$, and satisfies the bounds
  \begin{equation}\label{eq:spectral-bound}
    \frac{c_I}{\sqrt{c_a}} \le \frac{{\bf q}\mathcal{S}{\bf q}}{{\bf q}^T\mathcal{M}{\bf q}}
    \le \frac{c_b}{\sqrt{\alpha}} \quad \forall {\bf q} \in \mathbb{R}^{n_p + n_l}.
  \end{equation}
\end{lemma}
\begin{proof}
  It directly follows from \cref{eq:stab_bound_ah,eq:infsupbh} that
  \begin{equation}
    \label{eq:lowerboundnew}
    \frac{c_I}{\sqrt{c_a}}
    \le \inf_{(q_h,\xi_h)\in Q_h^{k-1}\times M_h^k} \sup_{\boldsymbol{v}_h \in \boldsymbol{V}_h^k}
    \frac{b_h((q_h,\xi_h), \boldsymbol{v}_h)}{a_h(\boldsymbol{v}_h,\boldsymbol{v}_h)^{1/2}
      \min(\beta, \beta_{\lambda})^{-1}\del[0]{\beta^2\norm{q_h}^2 + \beta_{\lambda}^2\norm[0]{\xi_h}_{\lambda}^2}^{1/2}}.
  \end{equation}
  Define ${\bf q}=[q^T\ \xi^T]^T$ and let us note that
  \begin{equation}
    \label{eq:weighted_qmnew}
    {\bf q}^T\mathcal{M}{\bf q} = 
    \frac{1}{\min(\beta^2, \beta_{\lambda}^2)}
    \del[0]{\beta^2\norm{q_h}^2 + \beta_{\lambda}^2\norm[0]{\xi_h}_{\lambda}^2}
    \geq 
    \norm{q_h}^2 + \norm[0]{\xi_h}_{\lambda}^2.
  \end{equation}
  It now follows from
  \cref{eq:stab_bound_ah,eq:bhextbounded,eq:weighted_qmnew} that
  \begin{equation}
    \label{eq:upperboundnew}    
    \frac{b_h((q_h,\xi_h), \boldsymbol{v}_h)}{a_h(\boldsymbol{v}_h,\boldsymbol{v}_h)^{1/2}
      \min(\beta, \beta_{\lambda})^{-1}\del[0]{\beta^2\norm{q_h}^2 + \beta_{\lambda}^2\norm[0]{\xi_h}_{\lambda}^2}^{1/2}}
    \le \frac{c_b}{\sqrt{\alpha}}.
  \end{equation}
  We can then express \cref{eq:lowerboundnew,eq:upperboundnew} in matrix forms as, respectively,
  \begin{equation*}
    \frac{c_I}{\sqrt{c_a}}
    \le \min_{{\bf q}}\max_{v\ne 0} \frac{{\bf q}\mathcal{B}v}{(v^TAv)^{1/2}({\bf q}\mathcal{M}{\bf q})^{1/2}}
    \quad\text{and}\quad
    \frac{{\bf q}\mathcal{B}v}{(v^TAv)^{1/2}({\bf q}\mathcal{M}{\bf q})^{1/2}}
    \le \frac{c_b}{\sqrt{\alpha}}.    
  \end{equation*}
  Following now identical steps as \cite[Section 3]{Pestana:2015} (see
  also the proof of \cite[Lemma 6]{Rhebergen:2018b}) it follows that
  \begin{equation*}
    \frac{c_I}{\sqrt{c_a}} \le \frac{{\bf q}\mathcal{S}{\bf q}}{{\bf q}^T\mathcal{M}{\bf q}}
    \le \frac{c_b}{\sqrt{\alpha}} \quad \forall {\bf q} \in \mathbb{R}^{n_p + n_l},
  \end{equation*}
  which completes the proof.
\end{proof}

\begin{remark}[Parameter values]
  \label{rem:omegaqmremark}
  The parameters $\omega_q$ and $\omega_m$ depend on the ratios
  $\beta/\min(\beta,\beta_{\lambda})$ and
  $\beta_{\lambda}/\min(\beta,\beta_{\lambda})$, respectively. In
  practice, these ratios are unknown; however, by construction we
  have either $\omega_q=1$ or $\omega_m=1$, so it is a matter of
  knowing which is smaller of $\beta$ and $\beta_\lambda$, and
  then choosing an estimate for the other parameter. In our simulations,
  we have observed that $\omega_q > \omega_m$ offers significantly better
  performance, suggesting that $\beta > \beta_{\lambda}$, in which
  case $\omega_m = 1$ and $\omega_q = (\beta/\beta_{\lambda})^2$. In
  \cref{fig:iters} we plot the iteration count versus $\omega_q$ to determine
  the best choice for $\omega_q$.
\end{remark}

\begin{remark}[$\omega_q=\omega_m=1$]
  Instead of \cref{lem:specEquivnew}, if the inf-sup condition
  \cref{eq:infsupbh} is replaced by
  \begin{equation}
    \label{eq:infsupbhI}
    c_I \del[0]{\norm{q_h}^2 + \norm{\xi_h}_{\lambda}^2}^{1/2} \le \sup_{\boldsymbol{v}_h \in \boldsymbol{V}_h^k}
    \frac{ b_h((q_h, \xi_h), \boldsymbol{v}_h) }{\norm{\boldsymbol{v}_h}_{1,h}}
    \quad \forall (q_h, \xi_h) \in Q_h^{k-1} \times M_h^k,
  \end{equation}
  it can be shown that the unweighted
  $\mathcal{M} = {\rm bdiag}(Q, M)$ is spectrally equivalent to
  $\mathcal{S}$ with the same bounds as in
  \cref{eq:spectral-bound}. This approach was indeed taken in
  \cite{Rhebergen:2018b} for an HDG discretization of the Stokes
  problem. This result, however, is in some sense less general than
  \cref{lem:specEquivnew} and corresponds to setting
  $\omega_q = \omega_m = 1$ in \cref{rem:omegaqmremark}. More
  importantly, numerical simulations show that choosing
  $\omega_q > \omega_m = 1$ can offer significant improvements in
  performance of the preconditioner.
\end{remark}

Finally, we introduce the preconditioner used in practice,
$\mathbb{P} = {\rm bdiag}(\tilde{A}, Q, M)$, where $\tilde{A}$ is an
operator that is spectrally equivalent to $A$. For example, algebraic
multigrid (AMG) methods are designed for $H^1$-like discrete systems
such as $A$, and we choose $\tilde{A}^{-1}$ to be the application of
several iterations of AMG to $A$. Then, the following corollary is an
immediate consequence of the spectral equivalence result between
$\mathcal{S}$ and $\mathcal{M}$ in \cref{lem:specEquivnew}, the
discrete inf-sup condition \cref{eq:infsupbh} (so that $\mathcal{B}$
is full rank), the fact that $A$ is symmetric and positive definite,
and \cite[Theorem 5.2]{Pestana:2015}.

\begin{corollary}
  \label{cor:optimal_prec}
  There exist positive constants $c_1,c_2,c_3,c_4$, independent of
  $h$, such that the eigenvalues $\rho$ of $\mathbb{P}^{-1}\mathbb{A}$
  satisfy $\rho \in [-c_1,-c_2] \cup [c_3,c_4]$.
\end{corollary}

\Cref{cor:optimal_prec} implies that MINRES applied to
$\mathbb{P}^{-1}\mathbb{A}{\bf x} = \mathbb{P}^{-1}{\bf f}$ will
converge to a given tolerance in a fixed number of iterations
regardless of the mesh size $h$, i.e., that $\mathbb{P}$ is an optimal
preconditioner for \cref{eq:hybridized}.

\section{Numerical examples}
\label{s:numex}

The extended system \cref{eq:hybridized} is implemented in both
Firedrake \cite{Rathgeber:2016} and MFEM \cite{mfem-library}
with solver support from PETSc and PETSc4py 
\cite{petsc-user-ref, petsc-web-page,Dalcin2011}
and \emph{hypre} \cite{hypre}. In all our simulations the
interior penalty parameter is chosen to be $\eta=4k^2$ for 2d and
$6k^2$ for 3d, with $k$ the polynomial degree in our approximation
spaces. When using AMG to solve the velocity block,
we primarily use the default PETSc \emph{hypre} parameters, which 
include classical interpolation and Falgout coarsening. The only 
modification is we use a strong threshold of $0.5$ for 2d $k=2$,
$0.25$ for 2d $k=4$, and 0.75 for 3d, choices arrived at via minor
parameter tuning. Note, higher-order and higher dimensional DG
discretizations of diffusion are known to be more challenging for
vanilla AMG, so scalable convergence using AMG to approximate the
velocity inverse requires more inner AMG iterations (e.g., we
use four inner AMG iterations for $k=2$ and ten AMG iterations
for $k=4$). However, if one were to substitute something like the
continuous/discontinuous preconditioning framework introduced in
\cite{pazner2020efficient} for high-order DG, one could likely get
away with two or three iterations per outer block preconditioning
iteration, without degrading convergence and independent of element
order. Moreover, although methods in \cite{pazner2020efficient}
rely on an additional projection to a continuous discretization, they
remain simpler and less invasive than a solver designed specifically
for $H({\rm div})$.

We consider the performance of the preconditioner $\mathbb{P}$ defined
in \cref{s:precon}, as well as the following block symmetric
Gauss--Seidel variant:
\begin{equation}
  \label{eq:ldu}
  \mathbb{P}^s = (\mathbb{A}_L + \mathbb{P})(\mathbb{P})^{-1}(\mathbb{A}_L^T + \mathbb{P}),
\end{equation}
where $\mathbb{A}_L$ is the strictly lower triangular part of the
block matrix $\mathbb{A}$. Note, \eqref{eq:ldu} is exactly an
approximate block LDU decomposition for the original Stokes system
expressed as a block $2\times 2$  operator, with velocity corresponding
to the (1,1)-block, and pressure and Lagrange multiplier (together)
representing the (2,2) block, with approximate (2,2) Schur
complement \eqref{eq:schur} given by $\mathcal{M}$ \eqref{eq:mass-prec}.

\subsection{Effect of weighted mass matrices}

Consider a 2d example in which the source term and boundary conditions
are chosen such that the exact solution on a domain $\Omega=(-1,1)^2$
is given by
$\boldsymbol{u} = [\sin(3 x)\sin(3 y), \cos(3 x) \cos(3 y)]$ and
$p = \sin(\pi x)\cos(\pi y)$. We first explore convergence of the
iterative solver as a function of constants $\omega_q$ and $\omega_m$,
with \Cref{fig:iters} showing number of MINRES iterations to relative
preconditioned residual tolerance $10^{-8}$ as a function of
$\omega_q/\omega_m$ (for ratio $<1$, this corresponds to $\omega_q=1$
and $\omega_m > 1$). These tests are run in parallel in the MFEM
library on an unstructured triangular mesh from Gmsh
\cite{Geuzaine:2009} with 69,632 elements, and total
degrees-of-freedom (DOFs) over all variables given by
$\{1358784, 2229504, 3309120, 4597632\}$ for 2nd--5th order finite
elements, respectively. Velocity solves are performed to a small
residual tolerance using AMG-preconditioned GMRES (for all practical
purposes ``exact''), while mass matrices are approximately inverted
using one iteration of parallel symmetric Gauss--Seidel. The mass
matrices are actually block-diagonal and could be inverted exactly
via, e.g., local LU decompositions, but the Gauss--Seidel iteration
provides identical convergence in the larger block preconditioning
iteration and does not require inverting many dense element matrices
(which, for higher order or 3d is a nontrivial computational cost).

First, note in \Cref{fig:iters} that the modified scalings of the mass
matrices, in particular $\omega_q > 1$, provide a 2--3$\times$
reduction in iteration count for 2nd--5th order finite
elements. Moreover, letting $\omega_m = 1$, we see that the best
constant $\omega_q > 1$ is moderately independent of element order,
with a choice of $\omega_q \approx 24-32$ providing near optimal
convergence in all cases (within 2-3 iterations of best observed
convergence; the lower end is slightly preferable for lower-order
elements, while the higher end preferable for higher-order
elements). Finally, we point out that the preconditioning is actually
\emph{super} robust in element order, with the minimum iteration count
slowly {decreasing} as element order increases: choosing the best
tested weight $\omega_q$ for each order, we have $\{29,26,23,21\}$
iterations for 2nd--5th order, respectively.

\begin{figure}[!htb]
  \centering
  \includegraphics[width=0.5\textwidth]{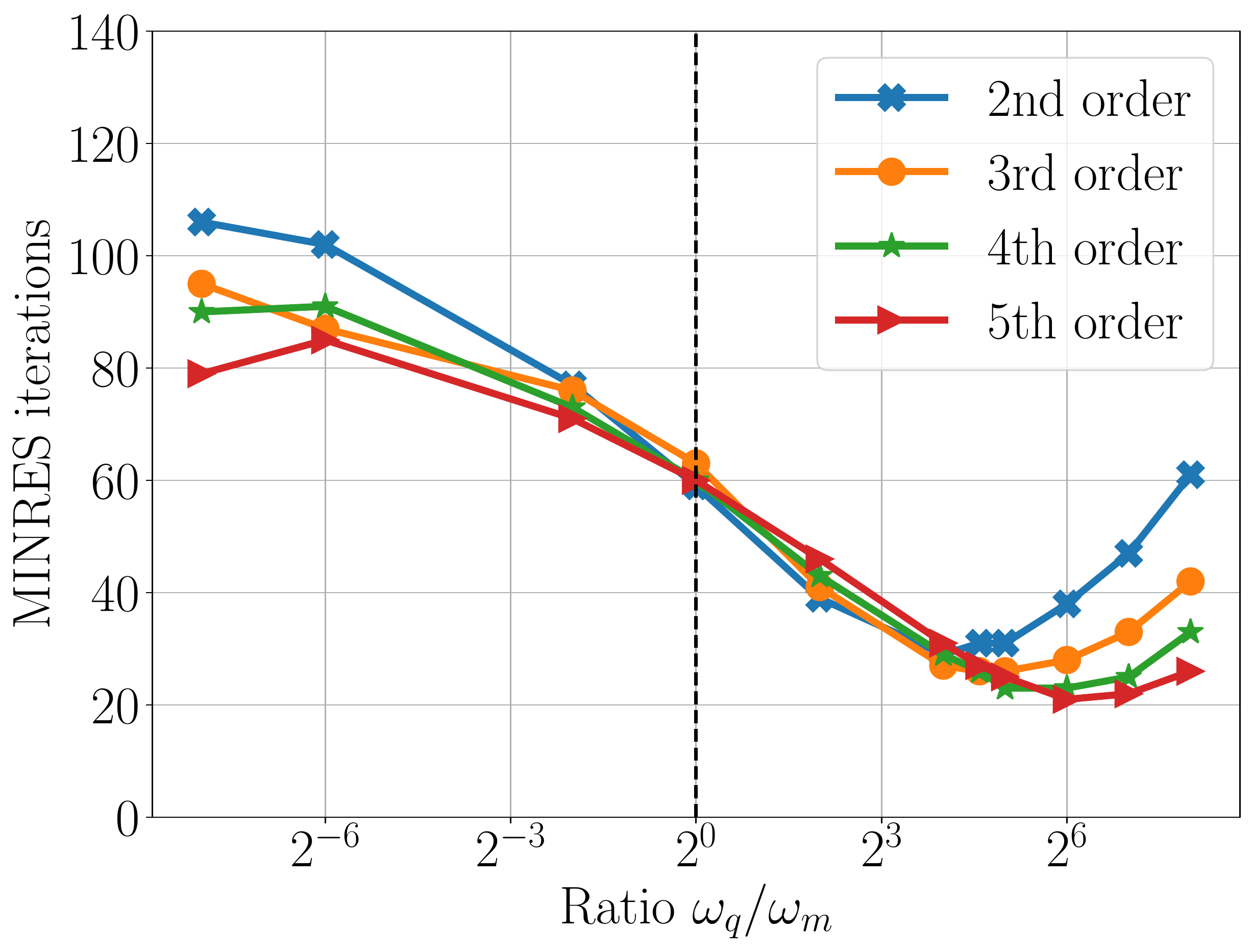}
  \caption{MINRES iterations to relative preconditioned residual tolerance
  $10^{-8}$ as a function of ratio of constants $\omega_q/\omega_m$ used
  in the mass-matrix approximation \eqref{eq:mass-prec} to the pressure
  Schur complement \eqref{eq:schur}. Black dotted line denotes the
  simplest case of $\omega_q=\omega_m=1$. On the right of this line,
  $\omega_m=1$, while on the left $\omega_q = 1$.}
      \label{fig:iters}
\end{figure}

\subsection{Two dimensional robustness test case}

Next we run 2d experiments in Firedrake to demonstrate $h$-robustness,
to compare exact and inexact solves for the velocity block, and to
compare iteration counts with a comparable preconditioner applied to a
true BDM discretization \cite{Cockburn:2007b,Wang:2007}. We choose the
source term and boundary conditions in \cref{eq:stokes} such that the
exact solution is given by
$\boldsymbol{u} = [\sin(\pi x)\sin(\pi y) + 2, \cos(\pi x) \cos(\pi y)
- 1]$ and $p = \sin(\pi x)\cos(\pi y)$ on the unit square
$\Omega=(0,1)^2$. In \cref{tab:stokes-results} we list the number of
iterations MINRES requires to reach a relative preconditioned residual
tolerance of $10^{-8}$ for the extended hybridized discretization,
using exact and approximate inner solves for $A$, $Q$, and $M$. For
approximate inner solvers, both mass matrices, $Q$ and $M$, are
approximated by one symmetric Gauss--Seidel relaxation, and $A$ is
approximated by four AMG iterations for $k=2$ and ten AMG iterations
for $k=4$. We note that both preconditioners (with exact and
approximate solves) are optimal in $h$, with only minor degradation in
convergence when moving from exact inner solves to approximate inner
solves. The modified constant $\omega_q > 1$ helps in all cases, and
is most notable for block-diagonal preconditioning with $k=4$,
reducing the iteration count $2.5\times$ compared with $\omega_q =
1$. In addition, using $\omega_q=24$, we again see the super
robustness in element order, with lower iteration counts in all cases
for $k=4$ compared with $k=2$.

\begin{table}[h!]
  \caption{Number of MINRES iterations required to reach a relative
    preconditioned residual tolerance of $10^{-8}$ for preconditioners
    $\mathbb{P}$ and $\mathbb{P}^s$ applied to \cref{eq:hybridized},
    with preconditioner constants $\omega_q = 1$ and $\omega_q = 24$.
    The number of elements in the mesh is given by $N^2 \times 2$.}
  \begin{center}
    \begin{tabular}{ccc||cc|cc||cc|cc}
      \hline
      &&& \multicolumn{4}{c||}{$\omega_q = 1$} & \multicolumn{4}{c}{$\omega_q = 24$} \\
      \hline
      &&& \multicolumn{2}{c|}{Exact} & \multicolumn{2}{c||}{Approx.}
        & \multicolumn{2}{c|}{Exact} & \multicolumn{2}{c}{Approx.} \\
      $k$ & $N$ & DOFs & $\mathbb{P}$ & $\mathbb{P}^s$ & $\mathbb{P}$ & $\mathbb{P}^s$
        & $\mathbb{P}$ & $\mathbb{P}^s$ & $\mathbb{P}$ & $\mathbb{P}^s$ \\
      \hline
      2 & 16  & 6144   & 136 & 60 & 138 & 64   & 64 & 39 & 64 & 43 \\
      2 & 32  & 24576  & 134 & 57 & 135 & 60   & 66 & 37 & 66 & 42 \\
      2 & 64  & 98304  & 134 & 50 & 134 & 56   & 66 & 34 & 66 & 38 \\
      2 & 128 & 393216 & 132 & 46 & 134 & 51   & 66 & 34 & 66 & 39 \\
      \hline
      4 & 16  & 24480   & 160 & 59 & 162 & 65   & 66 & 33 & 66 & 39 \\
      4 & 32  & 97600   & 160 & 54 & 162 & 60   & 64 & 32 & 64 & 38 \\
      4 & 64  & 389760  & 156 & 54 & 157 & 60   & 62 & 31 & 62 & 37 \\
      4 & 128 & 1557760 & 154 & 53 & 155 & 60   & 61 & 28 & 61 & 34 \\
      \hline
    \end{tabular}
  \end{center}
  \label{tab:stokes-results}
\end{table}

To compare, \cref{tab:stokes-results-bdm} now shows analogous
experiments applied to solve \cref{eq:DGmethod} without
hybridizing. Using similar analysis as in \cref{s:precon} for the
extended system \cref{eq:hybridized} it can be shown that if $R$ is
spectrally equivalent to $A$, then
$\tilde{\mathbb{P}} = {\rm bdiag}(R, M)$ is an optimal preconditioner
for \cref{eq:DGmethod}. Analogous to the hybridized case, we denote
the block diagonal and approximate LDU \eqref{eq:ldu} variations
$\tilde{\mathbb{P}}$ and $\tilde{\mathbb{P}}^s$, respectively.  The
main difference between solving \cref{eq:DGmethod} and
\cref{eq:hybridized} is the basis for the velocity space. Indeed for
\cref{eq:DGmethod} $u \in \mathbb{R}^{n_u}$ is the vector of
coefficients with respect to the basis for the BDM space
${\rm {\bf BDM}}_h^k$ while for \cref{eq:hybridized}
$u \in \mathbb{R}^{n_u}$ is the vector of coefficients with respect to
the basis for the DG space $\boldsymbol{V}_h^k$. As mentioned
previously, having the velocity space represented in $H({\rm div})$ is
significantly more challenging for black-box solvers such as AMG.

As for the hybridized case, in the preconditioner for
\cref{eq:DGmethod} we consider two choices of inner approximations: a
direct solver applied to $A$ and the pressure mass matrix, or
approximate inverses as used in \cref{tab:stokes-results}, with one
symmetric Gauss--Seidel iteration applied to the mass matrix and
$R^{-1} := \mathcal{O}(1)$ AMG iterations applied to $A$. From
\cref{tab:stokes-results-bdm} we note that if a direct solver is used
on $A$, then the preconditioners $\tilde{\mathbb{P}}$ and
$\tilde{\mathbb{P}}^s$ are optimal, and the iteration counts are
1.5--2.5$\times$ less than those observed for the extended system
\cref{eq:hybridized} in \cref{tab:stokes-results}.  However, no
convergence was obtained in any cases after switching to AMG to
precondition $A$. It should be noted that even with significant
parameter tuning, numerical tests indicated that even hundreds of
AMG-preconditioned Krylov iterations were unable to make a notable
reduction in residual when applied to the $H({\rm div})$-DG
discretization of the vector Laplacian. Thus, the fact that
block-preconditioning with AMG-approximate inverses did not converge
for BDM discretizations (see \cref{tab:stokes-results-bdm}) makes
sense, and is simply because AMG is not an effective preconditioner
for the $H({\rm div})$-DG discretization of the vector Laplacian.

\begin{table}[hbt!]
  \caption{Number of MINRES iterations required to reach a relative
    preconditioned residual tolerance of $10^{-8}$ for preconditioners
    $\tilde{\mathbb{P}}$ and $\tilde{\mathbb{P}}^s$ applied to \cref{eq:DGmethod}.
    Results are shown for order $k=2$ and $k=4$, both considered with
    exact inner solves and approximate inner solves. The number of
    elements in the mesh is given by $N^2 \times 2$. A $*$ means there
    was no convergence within 500 iterations.
    }
  \begin{center}
    \begin{tabular}{ccc|cc|cc}
      \hline
      &&& \multicolumn{2}{c|}{Exact} & \multicolumn{2}{c}{Approx.} \\
      $k$ & $N$ & DOFs & $\tilde{\mathbb{P}}$ & $\tilde{\mathbb{P}}^s$ & $\tilde{\mathbb{P}}$ & $\tilde{\mathbb{P}}^s$ \\
      \hline
      2 & 16  & 5472   & 32 & 16 & * & * \\
      2 & 32  & 21696  & 35 & 15 & * & * \\
      2 & 64  & 86400  & 33 & 15 & * & * \\
      2 & 128 & 344832 & 33 & 14 & * & * \\
      \hline
      4 & 16  & 16800   & 37 & 16 & * & * \\
      4 & 32  & 66880   & 37 & 16 & * & * \\
      4 & 64  & 266880  & 37 & 15 & * & * \\
      4 & 128 & 1066240 & 37 & 15 & * & * \\
      \hline
    \end{tabular}
  \end{center}
  \label{tab:stokes-results-bdm}
\end{table}

\subsection{Three dimensional robustness test case}

Last, we demonstrate that the proposed preconditioning naturally
extends to 3d problems, with even better convergence than the 2d
setting. For this we consider a domain $\Omega=(0,1)^3$ with the exact
solution given by
$\boldsymbol{u}=[\pi (\sin(\pi x)\cos(\pi y)-\sin(\pi x)\cos(\pi z)),
\pi (\sin(\pi y) \cos(\pi z)-\cos(\pi x) \sin(\pi y)), \pi (\cos(\pi
x) \sin(\pi z)-\cos(\pi y) \sin(\pi z))]$ and
$p=\sin(\pi x) \cos(\pi y) \sin(2 \pi z)$. Here we discretize in MFEM
using a 3d unstructured tetrahedral mesh generated with Gmsh
\cite{Geuzaine:2009} and apply the LDU version of block
preconditioning, $\mathbb{P}^s$, to the extended hybridized system.
For approximate solves, 14 AMG iterations are applied (because 3d DG
discretizations of diffusion are more challenging to AMG than 2d). We
again see the preconditioner is optimal in $h$, super robust in
element order (although AMG is not, hence the growth in approximate
iteration count between $k=2$ and $k=4$), and the weighting
$\omega_q = 32$ provides a consistent $3\times$ or more reduction in
iteration count compared with $\omega_q = 1$.

\begin{table}[h!]
  \caption{Number of MINRES iterations required to reach a relative
    preconditioned residual tolerance of $10^{-8}$ for preconditioner
    $\mathbb{P}^s$ applied to \cref{eq:hybridized} for a 3d problem,
    using order $k=2$ and $k=4$ elements, and preconditioner constants
    $\omega_q = 1$ and $\omega_q = 32$. $*$ indicates did not finish
    due to cluster time constraints.}
  \begin{center}
    \begin{tabular}{ccc||cc||cc}
      \hline
      & \multicolumn{2}{c||}{$\mathbb{P}^s$} & \multicolumn{2}{c||}{$\omega_q = 1$} & \multicolumn{2}{c}{$\omega_q = 32$} \\
      $k$ & Cells & DOFs & Exact & Approx. & Exact & Approx. \\
      \hline
      2 & 4192  & 195712 & 108 & 124 & 34 & 42 \\   %
      2 & 33536  & 1554176 & 88 & 104 & 30 & 39 \\   %
      2 & 268288  & 12387328 & 70 & 91 & 26 & 39 \\  %
      \hline
      4 & 4192  & 656960 & 98 & 156 & 27 & 59 \\   %
      4 & 33536  & 5226880 & 84 & 141 & 25 & 54 \\   %
      4 & 268288  & 41699840 & * & * & 23 & 54 \\   %
      \hline
    \end{tabular}
  \end{center}
  \label{tab:3d-stokes-results}
\end{table}

\section{Conclusions}
\label{sec:conclusions}

Using BDM elements for the velocity field in a DG discretization of
the Stokes equations results in a velocity field that is exactly
divergence-free. As a consequence, it can be shown that the
discretization is pressure-robust. Unfortunately, ``black-box'' solvers
perform poorly on discretizations using BDM elements and specialized
(custom) preconditioners are required for the efficient solution of
such systems.

As an alternative, one may hybridize the discretization, so that
the velocity lives in a ``friendlier'' standard $L^2$-DG space. 
Typically, hybridization allows one to eliminate all element degrees-of-freedom
and to obtain a global system for a Lagrange multiplier only. However,
since the BDM element is a nonconforming element for the Stokes
problem, elimination of element degrees-of-freedom is not
possible. Nevertheless, as we have shown in this manuscript, the
extended hybridized system is much better suited for ``black-box''
solvers. We have presented an optimal preconditioner for the
hybridized system of which each component is directly available in
most linear algebra packages. The preconditioner is robust in $h$,
super robust in element order, and naturally effective in 2d and 3d
(in fact, iteration counts are lower in 3d). Moreover, no custom
solvers are necessary.

Last, we point out that the hybridization extends naturally to
arbitrary Reynolds Navier--Stokes problems, and there exist fast
solvers for high-Reynolds velocity fields of the extended hybridized
system \cite{AIR,lAIR}. Future work will focus on developing block
preconditioning techniques for the extended hybridized system of
Navier--Stokes.

\section*{Acknowledgements}
SR gratefully acknowledges support from the Natural Sciences
and Engineering Research Council of Canada through the Discovery
Grant program (RGPIN-05606-2015). BSS was supported as a Nicholas C. Metropolis Fellow under the
Laboratory Directed Research and Development program of Los Alamos
National Laboratory. Los Alamos National Laboratory report number LA-UR-22-22682.

\bibliographystyle{plain}
\bibliography{references}

\begin{thebibliography}{10}

\bibitem{Arnold:2000}
D.~N. Arnold, R.~S. Falk, and R.~Winther.
\newblock Multigrid in ${H}(\text{div})$ and ${H}(\text{curl})$.
\newblock {\em Numer. Math.}, 85:197--217, 2000.

\bibitem{Arnold:1985}
Douglas~N Arnold and Franco Brezzi.
\newblock Mixed and nonconforming finite element methods: implementation,
  postprocessing and error estimates.
\newblock {\em ESAIM: Mathematical Modelling and Numerical Analysis},
  19(1):7--32, 1985.

\bibitem{petsc-user-ref}
Satish Balay, Shrirang Abhyankar, Mark~F. Adams, Jed Brown, Peter Brune, Kris
  Buschelman, Lisandro Dalcin, Victor Eijkhout, William~D. Gropp, Dmitry
  Karpeyev, Dinesh Kaushik, Matthew~G. Knepley, Dave~A. May, Lois~Curfman
  McInnes, Richard~Tran Mills, Todd Munson, Karl Rupp, Patrick Sanan, Barry~F.
  Smith, Stefano Zampini, Hong Zhang, and Hong Zhang.
\newblock {PETS}c users manual.
\newblock Technical Report ANL-95/11 - Revision 3.11, Argonne National
  Laboratory, 2019.

\bibitem{petsc-web-page}
Satish Balay, Shrirang Abhyankar, Mark~F. Adams, Jed Brown, Peter Brune, Kris
  Buschelman, Lisandro Dalcin, Victor Eijkhout, William~D. Gropp, Dinesh
  Kaushik, Matthew~G. Knepley, Lois~Curfman McInnes, Karl Rupp, Barry~F. Smith,
  Stefano Zampini, Hong Zhang, and Hong Zhang.
\newblock {PETS}c {W}eb page.
\newblock \url{http://www.mcs.anl.gov/petsc}, 2016.

\bibitem{Brezzi:1985}
Franco Brezzi, Jim Douglas, and L~Donatella Marini.
\newblock Two families of mixed finite elements for second order elliptic
  problems.
\newblock {\em Numerische Mathematik}, 47(2):217--235, 1985.

\bibitem{brown2000semicoarsening}
Peter~N Brown, Robert~D Falgout, and Jim~E Jones.
\newblock Semicoarsening multigrid on distributed memory machines.
\newblock {\em SIAM Journal on Scientific Computing}, 21(5):1823--1834, 2000.

\bibitem{Cockburn:2004b}
B.~Cockburn, G.~Kanschat, and D.~Sch\"otzau.
\newblock A locally conservative {LDG} method for the incompressible
  {N}avier--{S}tokes equations.
\newblock {\em Math. Comp.}, 74(251):1067--1095, 2004.

\bibitem{Cockburn:2007b}
B.~Cockburn, G.~Kanschat, and D.~Sch\"otzau.
\newblock A note on discontinuous {G}alerkin divergence-free solutions of the
  {N}avier--{S}tokes equations.
\newblock {\em J. Sci. Comput.}, 31(1/2):61--73, 2007.

\bibitem{Cockburn:2014b}
B.~Cockburn and F.~J. Sayas.
\newblock Divergence-conforming {HDG} methods for {S}tokes flows.
\newblock {\em Math. Comp.}, 83:1571--1598, 2014.

\bibitem{Dalcin2011}
Lisandro~D. Dalcin, Rodrigo~R. Paz, Pablo~A. Kler, and Alejandro Cosimo.
\newblock Parallel distributed computing using {P}ython.
\newblock {\em Advances in Water Resources}, 34(9):1124--1139, 2011.
\newblock New Computational Methods and Software Tools.

\bibitem{Pietro:book}
D.~A. Di~Pietro and A.~Ern.
\newblock {\em Mathematical Aspects of Discontinuous {G}alerkin Methods},
  volume~69 of {\em Math\'ematiques et Applications}.
\newblock Springer--Verlag Berlin Heidelberg, 2012.

\bibitem{Dios.2014}
Blanca Ayuso~de Dios, Franco Brezzi, L.~Donatella Marini, Jinchao Xu, and
  Ludmil Zikatanov.
\newblock {A Simple Preconditioner for a Discontinuous Galerkin Method for the
  Stokes Problem}.
\newblock {\em Journal of Scientific Computing}, 58(3):517--547, 2014.

\bibitem{Dobrev:2019}
V.~Dobrev, T.~Kolev, C.~S. Lee, V.~Tomov, and P.~S. Vassilevski.
\newblock Algebraic hybridization and static condensation with application to
  scalable ${H}(\text{div})$ preconditioning.
\newblock {\em SIAM J. Sci. Comput.}, 41(3):B425--B447, 2019.

\bibitem{Dobrev:2006}
V.~Dobrev, R.~D. Lazarov, P.~S. Vassilevski, and L.~T. Zikatanov.
\newblock Two-level preconditioning of discontinuous {G}alerkin approximations
  of second-order elliptic equations.
\newblock {\em Numer. Linear Algebra Appl.}, 13:753--770, 2006.

\bibitem{mfem-library}
V.~A. Dobrev, T.~V. Kolev, et~al.
\newblock {MFEM}: Modular finite element methods.
\newblock \url{http://mfem.org}, 2020.

\bibitem{hypre}
Robert~D Falgout and Ulrike~Meier Yang.
\newblock hypre: A library of high performance preconditioners.
\newblock In {\em International Conference on Computational Science}, pages
  632--641. Springer, 2002.

\bibitem{Geuzaine:2009}
C.~Geuzaine and J.-F. Remacle.
\newblock {G}msh: a three-dimensional finite element mesh generator with
  built-in pre- and post-processing facilities.
\newblock {\em International Journal for Numerical Methods in Engineering},
  79(11):1309--1331, 2009.

\bibitem{Greif:2010}
Chen Greif, Dan Li, Dominik Sch{\"o}tzau, and Xiaoxi Wei.
\newblock A mixed finite element method with exactly divergence-free velocities
  for incompressible magnetohydrodynamics.
\newblock {\em Computer Methods in Applied Mechanics and Engineering},
  199(45-48):2840--2855, 2010.

\bibitem{Guzman:2017}
Johnny Guzm{\'a}n, Chi-Wang Shu, and Fil{\'a}nder~A Sequeira.
\newblock {H}(div) conforming and {DG} methods for incompressible {E}uler’s
  equations.
\newblock {\em IMA Journal of Numerical Analysis}, 37(4):1733--1771, 2017.

\bibitem{Hiptmair.2007}
Ralf Hiptmair and Jinchao Xu.
\newblock {Nodal Auxiliary Space Preconditioning in H(curl) and H(div) Spaces}.
\newblock {\em SIAM Journal on Numerical Analysis}, 45(6):2483--2509, 2007.

\bibitem{Hong.2016}
Qingguo Hong, Johannes Kraus, Jinchao Xu, and Ludmil Zikatanov.
\newblock {A robust multigrid method for discontinuous Galerkin discretizations
  of Stokes and linear elasticity equations}.
\newblock {\em Numerische Mathematik}, 132(1):23--49, 2016.

\bibitem{Howell:2011}
J.~S. Howell and N.~J. Walkington.
\newblock Inf-sup conditions for twofold saddle point problems.
\newblock {\em Numer. Math.}, 118:663--693, 2011.

\bibitem{John:2017}
V.~John, A.~Linke, C.~Merdon, M.~Neilan, and L.~G. Rebholz.
\newblock On the divergence constraint in mixed finite element methods for
  incompressible flows.
\newblock {\em SIAM Rev.}, 59(3):492--544, 2017.

\bibitem{Kanschat:2015}
Guido Kanschat and Youli Mao.
\newblock Multigrid methods for {H}div-conforming discontinuous {G}alerkin
  methods for the {S}tokes equations.
\newblock {\em Journal of Numerical Mathematics}, 23(1):51--66, 2015.

\bibitem{Kanschat.2016}
Guido Kanschat and Youli Mao.
\newblock {Multiplicative Overlapping Schwarz Smoothers for Hdiv-Conforming
  Discontinuous Galerkin Methods for the Stokes Problem}.
\newblock In {\em Domain Decomposition Methods in Science and Engineering
  XXII}, Lecture Notes in Computational Science and Engineering, pages
  285--292, 2016.

\bibitem{Kanschat:2010}
Guido Kanschat and B{\'e}atrice Riviere.
\newblock A strongly conservative finite element method for the coupling of
  {S}tokes and {D}arcy flow.
\newblock {\em Journal of Computational Physics}, 229(17):5933--5943, 2010.

\bibitem{Lee:2017}
C.~H. Lee and P.~S. Vassilevski.
\newblock Parallel solver for {$H$(div)} problems using hybridization and
  {AMG}.
\newblock In {\em Domain Decomposition Methods in Science and Engineering
  XXIII}, pages 69--80. Springer, 2017.

\bibitem{Lehrenfeld:2016}
C.~Lehrenfeld and J.~Sch\"oberl.
\newblock High order exactly divergence-free hybrid discontinuous {G}alerkin
  methods for unsteady incompressible flows.
\newblock {\em Comput. Methods Appl. Mech. Engrg.}, 307:339--361, 2016.

\bibitem{AIR}
Thomas~A Manteuffel, Steffen Münzenmaier, John Ruge, and Ben Southworth.
\newblock Nonsymmetric reduction-based algebraic multigrid.
\newblock {\em SIAM Journal on Scientific Computing}, 41(5):S242--S268, 2019.

\bibitem{lAIR}
Thomas~A Manteuffel, John Ruge, and Ben~S Southworth.
\newblock Nonsymmetric algebraic multigrid based on local approximate ideal
  restriction ($\ell${AIR}).
\newblock {\em SIAM Journal on Scientific Computing}, 40(6):A4105--A4130, 2018.

\bibitem{Oyarzua:2014}
Ricardo Oyarz{\'u}a, Tong Qin, and Dominik Sch{\"o}tzau.
\newblock An exactly divergence-free finite element method for a generalized
  {B}oussinesq problem.
\newblock {\em IMA journal of numerical analysis}, 34(3):1104--1135, 2014.

\bibitem{pazner2020efficient}
Will Pazner.
\newblock Efficient low-order refined preconditioners for high-order
  matrix-free continuous and discontinuous galerkin methods.
\newblock {\em SIAM Journal on Scientific Computing}, 42(5):A3055--A3083, 2020.

\bibitem{Pestana:2015}
J.~Pestana and A.~J. Wathen.
\newblock Natural preconditioning and iterative methods for saddle point
  systems.
\newblock {\em SIAM Rev.}, 57(1):71--91, 2015.

\bibitem{Rathgeber:2016}
Florian Rathgeber, David~A Ham, Lawrence Mitchell, Michael Lange, Fabio
  Luporini, Andrew~TT McRae, Gheorghe-Teodor Bercea, Graham~R Markall, and
  Paul~HJ Kelly.
\newblock Firedrake: automating the finite element method by composing
  abstractions.
\newblock {\em ACM Transactions on Mathematical Software (TOMS)}, 43(3):1--27,
  2016.

\bibitem{Raviart:1977}
Pierre-Arnaud Raviart and Jean-Marie Thomas.
\newblock A mixed finite element method for 2-nd order elliptic problems.
\newblock In {\em Mathematical aspects of finite element methods}, pages
  292--315. 1977.

\bibitem{Rhebergen:2018a}
S.~Rhebergen and G.~Wells.
\newblock A hybridizable discontinuous {G}alerkin method for the
  {N}avier--{S}tokes equations with pointwise divergence-free velocity field.
\newblock {\em J. Sci. Comput.}, 76(3):1484--1501, 2018.

\bibitem{Rhebergen:2018b}
S.~Rhebergen and G.~N. Wells.
\newblock Preconditioning of a hybridized discontinuous {G}alerkin finite
  element method for the {S}tokes equations.
\newblock {\em J. Sci. Comput.}, 77(3):1936--1952, 2018.

\bibitem{roberts2004textbook}
Thomas~W Roberts, David Sidilkover, and RC~Swanson.
\newblock Textbook multigrid efficiency for the steady euler equations.
\newblock 2004.

\bibitem{stuben2001review}
Klaus St{\"u}ben.
\newblock A review of algebraic multigrid.
\newblock {\em Numerical Analysis: Historical Developments in the 20th
  Century}, pages 331--359, 2001.

\bibitem{Wang:2007}
J.~Wang and X.~Ye.
\newblock New finite element methods in computational fluid dynamics by
  {H}(div) elements.
\newblock {\em SIAM J. Numer. Anal.}, 45:1269--1286, 2007.

\end{thebibliography}
\end{document}